\documentclass[12 pt]{article}
\usepackage{amsmath,amsthm,amssymb}
\usepackage[all]{xy}
\theoremstyle{plain}

\newtheorem{thm}{Theorem}[section]

\newtheorem{lem}[thm]{Lemma}
\newtheorem{cor}[thm]{Corollary}

\newtheorem{lem-dfn}[thm]{Lemma-Definition}

\theoremstyle{definition}

\newtheorem{dfn}[thm]{Definition}

\newtheorem{conv}[thm]{Convention}
\newtheorem{nota}[thm]{Notation}

\newtheorem{rem}[thm]{Remark}

\DeclareMathOperator{\Div}{Div}
\DeclareMathOperator{\divi}{div}

\DeclareMathOperator{\Mob}{Mob}
\DeclareMathOperator{\Fix}{Fix}

\DeclareMathOperator{\mult}{mult}

\DeclareMathOperator{\Supp}{Supp}

\DeclareMathOperator{\Image}{Image}
\DeclareMathOperator{\Interior}{Int}
\DeclareMathOperator{\Effbar}{\overline{Eff}}
\DeclareMathOperator{\Eff}{Eff}
\DeclareMathOperator{\Bgg}{Big}

\newcommand{\oo}{\mathcal{O}}
\newcommand{\QQ}{\mathbb{Q}}
\newcommand{\RR}{\mathbb{R}}
\newcommand{\Rb}{\mathbf{R}}

\newcommand{\CC}{\mathbb{C}}
\newcommand{\Cc}{\mathcal{C}}
\newcommand{\Bc}{\mathcal{B}}
\newcommand{\Ec}{\mathcal{E}}
\newcommand{\TT}{\mathbb{T}}
\newcommand{\ZZ}{\mathbb{Z}}
\newcommand{\LL}{\mathbb{L}}
\newcommand{\Lc}{\mathcal{L}}

\newcommand{\GG}{\mathbb{G}}
\newcommand{\NN}{\mathbb{N}}
\newcommand{\Bb}{\mathbf{B}}
\newcommand{\Fb}{\mathbf{F}}
\newcommand{\xb}{\mathbf{x}}
\newcommand{\wb}{\mathbf{w}}

\newcommand{\eb}{\mathbf{e}}
\newcommand{\mb}{\mathbf{m}}

\title{Finite generation of adjoint rings after Lazi\'c: an
  introduction\footnote{This note will appear on the 
volume \emph{Classification of
Algebraic Varieties}, Carel Faber, Gerard van der Geer
and Eduard Looijenga Eds., EMS Publishing House}} 

\author{Alessio Corti\\Department of
    Mathematics\\ Imperial College London\\ London SW7 2AZ, UK}

\date{} 

\begin{document}
\maketitle

\tableofcontents

\section{Introduction}
\label{sec:intro}

This note is an introduction to all the key ideas of Lazi\'c's
recent proof of the theorem on the finite generation of adjoint rings
\cite{lazic09:_towar_mmp}. (The theorem was first proved in
\cite{Hb}.) I try to convince you that, despite technical issues that
are not yet adequately optimised, nor perhaps fully understood,
Lazi\'c's argument is a self-contained and 
transparent induction on dimension based on lifting lemmas and relying
on none of the detailed general results of Mori theory. On
the other hand, it is shown in \cite{CoLa} that all the fundamental
theorems of Mori theory follow easily from the finite generation
statement discussed here: together, these results give a new and more
efficient organisation of higher dimensional algebraic geometry.

The approach presented here is ultimately inspired by a close reading
of the work of Shokurov~\cite{MR1993750}, I mean specifically his
proof of the existence of \mbox{3-fold} flips. Siu was the first to
believe in the possibility of a direct proof of finite generation,
and believing that something is possible is, of course, a big part of
making it happen. All mathematical detail is taken from
\cite{lazic09:_towar_mmp}; my contribution is merely exegetic. I begin
with a few key definitions leading to the statement of the main
result.

\subsection{Basic definitions}
\label{sec:basic-definitions}

\begin{conv}
  Throughout this paper, I work with nonsingular projective varieties
  over the complex numbers.

  Let $V$ is a (finite dimensional) real vector space defined over the
  rationals.
  By a \emph{cone} in $V$ I always mean a convex cone, that is a
  subset $\Cc \subset V$ such that $0\in \Cc$ and:
  \begin{align*}
&t\geq 0,\;\mathbf{v}\in \Cc\; \Rightarrow t\mathbf{v} \in \Cc\,,\\
&\mathbf{v}_1, \; \mathbf{v}_2 \in \Cc \Rightarrow
\mathbf{v}_1+\mathbf{v}_2 \in \Cc .    
  \end{align*}
  A \emph{finite rational cone} is a cone $\Cc \subset V$ generated by
  a finite number of rational vectors. 
 
  If $U\subset V$ an affine subspace, then I denote by
  $U(\RR)$ and $U(\QQ)$ the sets of real and rational points of $U$. 

  The end of a proof of a statement, or the absence of a proof, is
  denoted by  \qed
\end{conv}

\begin{nota}
  I denote by $\Rb$ one of $\ZZ$, $\QQ$, $\RR$. If $X$ is a normal
  variety, I denote by $\Div_\Rb X$ the group of (Weil) divisors on
  $X$ with coefficients in $\Rb$, and by $\Div_\Rb^+ X$ the sub-monoid
  of effective divisors. In this note, I almost always work
  with actual divisors, \emph{not} divisors modulo linear
  equivalence. For instance, when I write $K_X$, I mean that I have
  chosen a specific divisor in the canonical class; the choice is
  made at the beginning and fixed throughout the discussion.
\end{nota}

\begin{nota}
  If $X$ is a normal variety and $D\in \Div_\Rb X$ is a divisor on
  $X$, I write:
\[
D=\Fix D + \Mob D,
\]
where $\Fix D$ and $\Mob D$ are the \emph{fixed} and the \emph{mobile
  part} of $D$. The definition makes sense even when $D$ is not integral or
  effective. Indeed the sheaf $\oo_X(D)$ is defined as
\[
\Gamma\bigl(U,\oo_X(D)\bigr) = \bigl\{f\in k(X) \mid \divi_U
f+D_{|U}\geq 0 \bigr\}
\] 
and then, by definition:
\[
\Mob D = \sum m_E\, E 
\quad
\text{where}
\quad
m_E = -\inf \{\mult_E f \mid f\in H^0\bigl(X,\oo_X(D)\bigr) \} .
\]

In applications $D$ is almost always integral and effective. If $D$ is
not integral, then the definition says that
$\oo_X(D)=\oo_X\bigl(\lfloor D \rfloor\bigr)$; if $D$ is integral, then
$\Fix D = \Fix |D|$ is the fixed part of the complete linear system
$|D|$. 
\end{nota}

Throughout this paper, I use without warning the following elementary
fact, often called Gordan's lemma.

\begin{lem}
  \label{lem:8}
  Let $\Cc \subset \RR^r$ a finite rational cone. The monoid $\Lambda
  = \Cc \cap \ZZ^r$ is finitely generated. \qed
\end{lem}

\begin{dfn}
  \label{dfn:1}
  Let $X$ be a nonsingular projective variety and $\Lambda = \Cc \cap
  \ZZ^r$ where $\Cc \subset \RR^r$ is a finite rational cone.
  \begin{enumerate}
  \item A \emph{divisorial ring} on $X$ is a $\Lambda$-graded ring of
  the form
\[
R(X; D)=\bigoplus_{\lambda \in \Lambda}H^0\bigl(X; D(\lambda)\bigr)
\quad
\text{where}
\quad
D \colon \Lambda \to \Div_\Rb X
\]
is a map such that $M(\lambda)=\Mob D(\lambda)$ is superadditive,
i.e., $M(\lambda_1+\lambda_2) \geq M(\lambda_1)+M(\lambda_2)$ for all
$\lambda_1, \, \lambda_2 \in \Lambda$. The map $D \colon \Lambda \to
\Div_\Rb X$ is called the \emph{characteristic system} of the ring.
When I wish to emphasise the grading by $\Lambda$, I write $R(X;
\Lambda)$ instead of $R(X; D)$.   
  \item A divisorial ring $R(X;D)$ is \emph{adjoint} if in addition
    \begin{itemize}
    \item $D\colon \Lambda \to \Div_\QQ X$ is \emph{rational PL},
      i.e., it is the restriction of a rational piecewise linear map
      that, abusing notation, I still denote by
\[
D \colon \Cc \to \Div_\RR X  .
\] 
      This means that there is a finite decomposition $\Cc =
      \cup_{i=1}^m \Cc_i$ into finite rational cones $\Cc_i\subset
      \RR^r$ such that $D_{|\Cc_i}$ is (the restriction of) a rational
      linear function. 
    \item there is a rational PL function $r\colon \Cc \to \RR_+$ and
    for all $\lambda \in \Lambda$
    $D(\lambda)=r(\lambda)\bigl(K_X+\Delta(\lambda) \bigr)$ with
    $\Delta(\lambda) \geq 0$.
    \end{itemize}
    \item An adjoint ring is \emph{big} if there is an ample
  $\QQ$-divisor $A$ and
\[
\Delta (\lambda) = A + B(\lambda)
\] 
  with all $B(\lambda) \geq 0$.
\item A big adjoint ring is \emph{klt (dlt)} if there is a fixed
  simple normal crossing (snc) divisor $\sum_{j=1}^r B_j\subset X$
  such that all $\Supp B(\lambda)\subset \cup_{j=1}^rB_j$ and all $\bigl(X, B
  (\lambda)\bigr)$ are klt (dlt).   
  \end{enumerate}
\end{dfn}

\begin{rem}
  \label{rem:6}
  \begin{enumerate}
  \item If $D\colon \Lambda \to \Div^+_\QQ X$ is superadditive and rational
  PL, then it is also concave.
  \item If $R(X;D)$ is an adjoint ring then in particular: $\Delta
    \colon \Cc \to \Rb_+$ is \emph{homogeneous of degree $0$}; that
    is, $\Delta (t\wb)=t\Delta (\wb)$ for all $t\in \RR_+$,
    $\wb \in \Cc$.
  \item Note that in the definition of dlt (klt) big adjoint ring, all
    divisors in sight are contained in a fixed snc divisor
    $\sum_{j=1}^r B_j$. In this context, a (big) adjoint ring is dlt
    (klt) if and only if
\[
B (\lambda)=\sum_{j=1}^r b_j(\lambda) B_j
\] 
where all $0\leq b_j(\lambda)\leq 1$ ($<1$) for all $\lambda \in
\Lambda$, $i=j,\dots, r$. In this paper we never need the definitions,
results and techinques of the general theory of singularities of
pairs.
  \end{enumerate}  
\end{rem}

\subsection{The main result}
\label{sec:main-result}

\begin{thm}[Theorem~$A$]
  \label{thm:8}
  \cite{Hb} Let $X$ be nonsingular projective, and $R=R(X;D)$ a dlt
  big adjoint ring on $X$. Assume, in addition, that $D \colon \Lambda
  \to \Div_\QQ^+ X$, that is, $D(\lambda) \geq 0$ for all $\lambda \in
  \Lambda$. Then $R=R(X;D)$ is finitely generated.
\end{thm}

\begin{rem}
  \label{rem:10}
  The additional assumption can be removed. The statement is written
  here in the form that best suits the logic of proof described below in
  section~\ref{sec:logic-proof}. Once theorem~$A$, and theorems~$B$ and~$C$
  of section~\ref{sec:logic-proof}, are proved, then the additional
  assumption is removed by a straightforward application of theorem~$C$.   
\end{rem}

\begin{cor}
  \label{cor:1}
  If $X$ is nonsingular projective of general type, then the canonical
  ring $R(X,K_X)$ is finitely generated. \qed
\end{cor}

\begin{rem}
  \label{rem:4}
  In fact, by work of Fujino and Mori \cite{MR1863025}, the results
  here imply the stronger statement that if $X$ is nonsingular, then
  the canonical ring of $X$ is finitely generated. I don't know if
  theorem~\ref{thm:8} can similarly be strengthened: it would be
  \emph{extremely} useful--see \cite{CoLa}--if it could.
\end{rem}

\begin{rem}
  \label{rem:5}
\begin{itemize}
\item This theorem is proved in \cite{Hb}. That proof uses all that is
  known about the minimal model program; in particular, I mention
  \cite{MR1225842,MR1993750,MR2242631,MR2352762,Hb,Hc}.
\item The proof by Lazi\'c is a self-contained induction on dimension
  based on lifting lemmas \cite{MR1660941},
  \cite[Chapter~5]{MR2352762}, \cite{Hc}, etcetera. On the other hand
  it is shown in \cite{CoLa} that theorem~A readily implies all the
  fundamental theorems of Mori theory.
\end{itemize}
\end{rem}

\subsection{The logic of the proof}
\label{sec:logic-proof}

In \cite{lazic09:_towar_mmp} theorem~$A$ is proved by a bootstrap
induction on dimension together with two other theorems that I state
shortly following some preparations.

\subsubsection*{Asymptotic fixed part}
\label{sec:basic-definitions-1}

I summarise some facts on asymptotic invariants of divisors, mostly
following \cite{MR2282673}.

For a projective normal variety $X$, I denote by $\Eff (X;\Rb) \subset
N^1(X;\Rb)$ the cone of (numerical equivalence classes of) effective
divisors with coefficients in $\Rb$, and by $\Effbar (X;\RR) \subset
N^1(X;\RR)$ the cone of \emph{pseudo-effective} divisors, that is, the
closure of the cone of effective divisors (it only makes sense to do
this with real coefficients).  Similarly, I denote by $\Bgg
(X;\Rb)\subset N^1(X; \Rb)$ the cone of big divisors; $\Bgg (X; \RR)$
is the interior of $\Effbar (X;\RR)$.

If $D\in \Div_\QQ X$ is a $\QQ$-divisor, then the \emph{stable base
  locus} of $D$ is the subset
\[
\Bb (D) = \bigcap_{0<p\in \ZZ,\; pD \in \Div_\ZZ X} Bs |p D| \subset X 
\]
(if $|pD| = \emptyset$ for all $0<p\in \ZZ$, then I say by convention
that $\Bb (D)=X$). 

If $D\in \Div_\QQ X$ is a big $\QQ$-divisor, then the \emph{asymptotic fixed
  part} of $D$ is the divisor
\[
\Fb (D)=\inf_{0<n\in \ZZ,\; nD \in \Div_\ZZ X}\frac1{n} \Fix nD \in
\Div_\RR^+ X  .
\]
It is clear that $\Fb(-)$ is a degree $1$ homogeneous convex function of the
numerical equivalence class of $D$; thus, it extends to a degree $1$
homogeneous convex function:
\[
\Fb \colon \Bgg (X; \RR)\to \Div^+_\RR X .
\]

In his book \cite{MR2104208}, Nakayama defines, as follows, a
canonical extension of this function to the closure $\overline{\Bgg}
(X;\RR)=\Effbar (X; \RR)$:
\[
\Fb (D)= \lim_{\varepsilon \to 0^+} \Fb (D+\varepsilon A)
\quad
\text{for}
\quad
D \in \Effbar (X; \RR)
\]
where $A$ is an ample divisor (the definition is independent of the
choice of $A$). A subtle point is that $\Fb (-)$ is continuous on
$\Bgg (X; \RR)$ (because it is convex), but not necessarily on
$\Effbar (X;\RR)$.

The paper \cite{MR2282673} promotes the view that certain asymptotic
invariants defined on $\Effbar (X;\RR)$, for instance the asymptotic
fixed part, are reasonably well-behaved under surprisingly general
conditions. The assertions that follow demonstrate how these
invariants are \emph{much better behaved} on the subcone of \emph{adjoint
  divisors}.

\subsubsection*{General setup}
\label{sec:general-setup}

In what follows, $X$ is a nonsingular projective variety, $A$ an ample
$\QQ$-divisor on $X$, and $\sum_{j=1}^r B_j$ a snc divisor on $X$. I
denote by $\RR^r\cong V\subset \Div_\RR X$ the vector subspace spanned
by the components $B_j$. I write:
\begin{align*}
  \Lc_V&=\{B \in V \mid K_X+B \; \text{is log canonical} \}=
                              \prod_{j=1}^r [0,1]B_j;\\
  \Ec_{V,A}&=\{B \in \Lc_V \mid K_X+A+B \in \Effbar X\}
                                      \subset \Lc_V .
\end{align*}
In addition, if $S$ is a component of
$\sum_{j=1}^r B_j$, I write:
\begin{align*}
  \Bc_{V,A}^S&=\{B \in \Ec_{V,A}\mid S \not \subset \Bb (K_X+A+B)\}
                                    \subset \Ec_{V,A};\\
  \Bc_{V,A}^{S=1}&=\{B=S+B^\prime \in \Ec_{V,A}\mid S \not \subset 
  \Bb (K_X+A+S+B^\prime)\} .
\end{align*}

\subsubsection*{Statements}
\label{sec:statements}

\begin{thm}[Theorem~$B$]
  \label{thm:7}
  $\Bc^{S=1}_{V,A}$ is a rational polytope; moreover:
\[
\Bc^{S=1}_{V,A}= \{B=S+B^\prime \in \Ec_{V,A} \mid \mult_S
\Fb(K_X+A+B)=0\} .
\] \qed
\end{thm}

\begin{thm}[Theorem~$C$]
  \label{thm:9}
  $\Ec_{V,A}$ is a rational polytope. More precisely, $\Ec_{V,A}$ is
  the convex hull of finitely many rational vectors $K_X+A+B_i$
  where, for all $i$: $B_i\geq 0$ is a $\QQ$-divisor, and there is a
  positive integer $p_i> 0$ such that $|p_i(K_X+A+B_i)|\not =
  \emptyset$. \qed
\end{thm}

\begin{thm}[Theorem~$B^+$]
  \label{thm:4}
  $\mathcal{B}_{V,A}^S$ is a rational polytope. Furthermore one can
  say the following. There is an integer $r>0$ such that:
  \begin{itemize}
  \item Suppose that $B\in \mathcal{L}_V$ and no component of $B$ is
    in $\Bb (K_X+A+B)$. If $p(K_X+A+B)$ is an integral divisor, then no
    component of $B$ is in $\Fix \bigl(rp(K_X+A+B)\bigr)$;
  \item Suppose that $B\in \mathcal{E}_{V,A}$. If $p(K_X+A+B)$ is
    an integral divisor, then $|rp(K_X+A+B)|\not = \emptyset$.
  \end{itemize}
\end{thm}

\subsubsection*{The logic of the proof}
\label{sec:leitfaden}

Here is a table of logical dependencies in Lazi\'c's paper where,
e.g., $A_n$ signifies `Theorem~$A$ in dimension $n$:' 

\begin{align}
 A_{n}+C_{n} & \Rightarrow B_n^+  \\ 
 A_{n-1} + C_{n-1} & \Rightarrow B_n \\
 A_{n-1} + C_{n-1}+B_n & \Rightarrow C_n \\
 B_{n-1}^+ + B_n & \Rightarrow A_n  
\end{align}

In this note, I outline all the key steps and ideas in the proof
of the last implication $B_{n-1}^+ +B_n \Rightarrow A_n$,
stopping somewhat short of a complete proof. The other implications have
similar and easier proofs. For a direct analytic proof of a weaker
form of theorem~$C$, see also \cite{Paun}.

\subsection{The key ideas of the proof}
\label{sec:key-ideas-proof}

The proof is based on two key ideas that I explain in a bit more
detail below and then fully in the rest of the note. The first is to
prove finite generation of strictly dlt rings (see below) by
restriction to a boundary divisor using lifting lemmas and induction
on dimension. The second idea is what I call below the ``main
construction.''  Starting with a klt big adjoint ring with
characteristic system $D\colon \Cc \to \Div_\RR X$, I inflate the cone
$\Cc$ and semigroup $\Lambda$ to a larger cone $\Cc^\prime$ and
semigroup $\Lambda^\prime$, and then ``chop'' into finitely many
smaller $\Cc_j$ and $\Lambda_j$ such that the rings $R(X; \Lambda_j)$
are strictly dlt. This is a version of constructions that are
ubiquitous in the proofs of all the fundamental theorems of Mori
theory. Finite generation of $R(X; \Lambda)$ follows easily from
finite generation of the $R(X;\Lambda_j)$.

\subsubsection*{Restriction of strictly dlt rings}
\label{sec:restr-strictly-lt}

\begin{dfn}
  \label{dfn:9}
I say that a dlt big adjoint ring $R(X; D)$ with characteristic
system
\[
D(\lambda)=r(\lambda)\bigl(K_X+\Delta(\lambda)\bigr), 
\quad
\text{where}
\quad
\Delta(\lambda)=A+B(\lambda)
\]
is \emph{strictly dlt} if there is a prime divisor $S$ that
appears in all $B(\lambda)$ with multiplicity one:
\[
B(\lambda)=S+B^\prime(\lambda).
\]

I say that $R(X;D)$ is \emph{plt} if it is strictly dlt and all
$\bigl(X,S+B(\lambda)\bigr)$ are plt.    
\end{dfn}

When $R(X;D)$ is a strictly dlt adjoint ring, it is natural to
want to study the restriction homomorphisms:
\[
\rho_\lambda \colon H^0\Bigl(X;r(\lambda)\bigl(K_X+A+S+B(\lambda)\bigr) 
\Bigr) \to H^0\Bigl(S;r(\lambda)\bigl(K_S+\Omega(\lambda)\bigr)\Bigr) .
\]
where $\Omega (\lambda)= \bigl(A+B(\lambda)\bigr)_{|S}$.
The $\rho_\lambda$ are not surjective, but, with small additional
assumptions, lifting lemmas give us a good control on the images.

\begin{thm}
  \label{thm:12}
  Assume theorems~$B_{n-1}^+$ and~$B_n$; let $\dim X=n$ and let $R(X; D)$
  be a strictly dlt big adjoint ring on $X$. The \emph{restricted ring:}
\[
R_S(X;D)=\sum_{\lambda\in \Lambda} \Image (\rho_\lambda) \subset R(S; D(\lambda)_{|S})
\]
is a klt big adjoint ring.
\end{thm}

\begin{rem}
  \label{rem:11}
  In fact, $R_S(X;D)$ is a klt big adjoint ring not on $S$, but on some
  birational model $T\to S$. This technicality is relevant in
  the proof of the theorem, but it is otherwise unimportant. 
\end{rem}
 
The techniques of Lazi\'c's proof of theorem~\ref{thm:12} are subtle
but generally well understood by the experts.  I sketch the key ideas
in section~\ref{sec:restr-strictly-lt-1} below.

\subsubsection*{The main construction}
\label{sec:reduct-strictly-lt}

In section~\ref{sec:reduct-strictly-lt-1}, I give a complete proof
that theorem~\ref{thm:12} implies theorem~$A$. Roughly speaking, here
is the outline: I want to show that a given dlt adjoint ring
$R(X;\Lambda)$, where $\Lambda=\Cc \cap \ZZ^r\subset \RR^r$,
satisfying the additional assumption of theorem~$A$, is finitely
generated.

First, I inflate $\Cc$ to a larger cone $\Cc \subset \Cc^\prime
\subset \RR^r$ and extend $D\colon \Lambda \to \Div^+_\QQ X$ to an
appropriate $D^\prime \colon \Lambda^\prime = \Cc^\prime \cap \ZZ^r
\to \Div^+_\QQ X$.

Next, I construct a decomposition into subcones:
\[
\Cc^\prime = \cup_{j=1}^r \Cc_j
\] 
such that, for all $j=1,\dots, r$, writing:
\[
D_j=D^\prime_{|\Lambda_j}\colon \Lambda_j=\Cc_j \cap \ZZ^r \to \Div^+_\QQ ,
\]
the ring:
\[
R_j=R(X;\Lambda_j) 
\quad
\text{is strictly dlt with}
\quad
D_j(\lambda)=r_j(\lambda)\bigl( B_j+B_j^\prime (\lambda)\bigr).
\]

Finally, each of the $R_j$ has a surjective restriction homomorphism to a
restricted ring:
\[
\rho_j \colon R_j(X;\Lambda_j)\to R_{B_j}(X;\Lambda_j),
\]
and a relatively straightforward argument shows, assuming---as I
may by theorem~\ref{thm:12} and induction on dimension---that the
restricted rings $R_{B_j}(X;\Lambda_j)$ are finitely generated, that the
ring $R(X;\Lambda^\prime)$ also is finitely generated, and then
ultimately so is the ring $R(X;\Lambda)$. 

The construction is explained in detail in
section~\ref{sec:reduct-strictly-lt-1}.

\subsection*{Acknowledgements}
\label{sec:acknowledgements}

I thank Paolo Cascini, J\"urgen Hausen, Anne-Sophie Kaloghiros, Vlad
Lazi\'c, James M\textsuperscript{c}Kernan and the referee for valuable
comments and suggestions.

\section{Natural operations with divisorial rings}
\label{sec:natural-operations}

I briefly discuss the behaviour of divisorial and adjoint rings under
natural operations. These properties are elementary and mostly
well-known \cite{MR2207875, ivan0}.

\subsection{Veronese subrings}
\label{sec:veronese-subrings}

\begin{dfn}
  \label{dfn:6}
  If $R=\oplus_{\lambda\in \Lambda}R_\lambda$ is a $\Lambda$-graded
  ring (e.g., $R$ could be a divisorial ring), $\LL \subset \ZZ^r$ is a
  finite index subgroup and $\Lambda^\prime = \Lambda \cap \LL$, then
  I say that 
\[
R^\prime = \bigoplus_{\lambda \in \Lambda^\prime} R_\lambda
\subset R
\]  
  is a \emph{Veronese subring.}
\end{dfn}

\begin{rem}
  \label{rem:7}
  If $R^\prime \subset R$ is a Veronese subring, then $R$ is finitely
  generated if and only if $R^\prime$ is. Indeed, $R^\prime\subset R$
  is the ring of invariants under the action of the finite group
  $\GG=\ZZ^r/\LL$, so the statement is a special case of a well-known
  theorem of E.\ Noether.
\end{rem}

\subsection{Inflating}
\label{sec:chopping-inflating}

A more general version of the next lemma can be found in
\cite[Proposition~1.1.6]{ivan0}. 

\begin{lem}
  \label{lem:2}
  Consider an inclusion of finite rational cones: 
\[
\Cc \subset \Cc^\prime \subset \RR^r ;
\] 
and write $\Lambda^\prime=\Cc^\prime \cap \ZZ^r$, $\Lambda=\Cc \cap
\ZZ^r$.  Let $R^\prime =\oplus_{\lambda\in \Lambda^\prime}
R^\prime_\lambda$ be a $\Lambda^\prime$-graded ring, and write
\[
R = \bigoplus_{\lambda\in \Lambda}R^\prime_\lambda .
\]
If $R^\prime$ is finitely generated, then so is $R$.
\end{lem}

\begin{proof}
  This is elementary and well-known, so I only give a very brief sketch
  of the proof. The cone $\Cc \subset \Cc^\prime$ is cut out by finitely
  many rational hyperplanes; working one hyperplane at a time, I may
  assume that
\[
\Cc = \{\wb\in \Cc^\prime \mid f(\wb)\geq 0\},\quad
\text{for a linear map}\;
f\colon \RR^r \to \RR
\]
with $f(\ZZ^r)\subset \ZZ$. Now the $\Lambda^\prime$-grading on
$R^\prime$ means that $R^\prime$ has a $\TT=\CC^{\times\,r}$-action,
and $f\colon \ZZ^r\to \ZZ$ corresponds to a $1$-parameter $\CC^\times
\to \TT$, in turn endowing $R^\prime$ with a $\ZZ$-grading, and then,
tautologically:
\[
R = R^\prime_+ = \bigoplus_{n\geq 0} R^\prime_n .
\]
Now $R^\prime_+ \subset R^\prime$ is finitely generated, because it is
the subring of invariants for the action of the reductive group
$\CC^\times$ on
\[
R^\prime[z]=\bigoplus_{n\geq 0}z^nR^\prime=\bigoplus_{n\in \ZZ_{\geq
    0},\; m\in \ZZ} z^nR^\prime_m
\] 
acting on $z^nR^\prime_m$ with weight $-n+m$. 
\end{proof}

\subsection{Injective characteristic systems}
\label{sec:inject-char-syst}

In general, the characteristic system $D\colon \Lambda \to \Div_\QQ X$
of a divisorial ring is not injective.

\begin{lem}
  \label{lem:9}
  Consider a characteristic system $D\colon \Lambda \to \Div_\QQ X$
  where $\Lambda =\Cc \cap \ZZ^r$ for a finite rational cone $\Cc
  \subset \RR^r$. Assume that $D$ is the restriction of a rational
  linear function, still denoted $D\colon \RR^r \to \Div_\RR
  X$. Write $\overline{\Cc} =D(\Cc)\subset \Div_\QQ X$, the image of
  $\Cc$ under $D$, and $\overline{\Lambda}=\overline{\Cc}\cap \Div_\ZZ
  X\subset \Div_\ZZ X$. Then $R(X;\Lambda)$ is finitely generated if
  and only if $R(X; \overline{\Lambda})$ is finitely generated.
\end{lem}

\begin{proof}
  A simple application of all the above. 
\end{proof}

\subsection{$\QQ$-linear equivalence}
\label{sec:qq-line-equiv}

\begin{dfn}
  \label{dfn:7}
  Let $X$ be a projective normal variety. Denote by $\Div^0_\Rb
  X\subset \Div_\Rb X$ the subgroup of divisors that are
  $\Rb$-linearly equivalent to $0$. Two characteristic systems on $X$:
\[
D\colon \Lambda \to \Div_\QQ X
\; \text{and}
\quad
D^\prime \colon \Lambda \to \Div_\QQ X
\]
 are $\QQ$-\emph{linearly equivalent} if there exists an additive map
 $\divi \colon \Lambda \to \Div_\QQ^0 X$ such that
\[
D(\lambda)=D^\prime (\lambda) + \divi (\lambda)
\quad \text{for all}\;
\lambda \in \Lambda.
\]
\end{dfn}

\begin{rem}
  \label{rem:8}
  If $D$ and $D^\prime$ are $\QQ$-linearly equivalent, then $R(X;D)$
  and $R(X;D^\prime)$ have isomorphic Veronese subrings. In
  particular, one is finitely generated if and only if the other is. 
  
  In some arguments, I use this device to replace the ample
  $\QQ$-divisor $A$ by a $\QQ$-linearly
  equivalent $\QQ$-divisor $A^\prime$ such that 
  $A^\prime$ is ``general,'' in the sense that $A^\prime\geq 0$, the
  coefficients of $A^\prime$ are as small as I care for them to be,
  and $A^\prime $
  meets every divisor and locally closed locus in sight as generically
  as possible.
\end{rem}

\begin{lem}
  \label{lem:4}
  Let $X$ be nonsingular projective and $R(X;D)$ a big adjoint ring on $X$.
  \begin{itemize}
  \item If $R(X;D)$ is dlt, then there exists a $\QQ$-linearly
    equivalent system $D^\prime$ such that $R(X;D^\prime)$ is a klt big
    adjoint ring.
  \item If $R(X;D)$ is strictly dlt, then there exists a $\QQ$-linearly
    equivalent system $D^\prime$ such that $R(X;D^\prime)$ is a plt big
    adjoint ring.
  \end{itemize}
\end{lem}

\begin{proof}[Sketch of Proof] The idea is, of course, to ``absorb''
    into $A$ a small amount of $B(\lambda)$ where it has coefficient $1$.
    I briefly discuss a very special case that illustrates the key 
    issue. Assume that $\Lambda = \NN^2= \NN
    \mathbf{e}_1+\NN\mathbf{e}_2$ and
    \begin{align*}
      D(\mathbf{e}_1) & =K_X+A+S_1,\\
      D(\mathbf{e}_2) & =K_X+A+S_2 
    \end{align*}
    where $S_1$, $S_2$ are smooth and meet transversally. The ring
    $R(X;D)$ is dlt. 

    For $N\gg 0$ we can write:
    \begin{align*}
      NA \sim S_1+T_2 \sim S_2+T_1
    \end{align*}
    where $S_1+S_2+T_1+T_2$ is a snc divisor. Choose a rational
    function $\varphi \in k(X)$ such that $-S_1+T_1=-S_2+T_2+\divi_X
    \varphi$. For $0\ll \varepsilon \ll 1$ we have:
    \[
    D=-\frac{\varepsilon}{N}S_1 +\frac{\varepsilon}{N}T_1
     =-\frac{\varepsilon}{N}S_2 +\frac{\varepsilon}{N}T_2 +
     \frac{\varepsilon}{N} \divi_X (\varphi) .
    \]
    Then $A^\prime = A -D$ is ample and, setting $B_i=(\varepsilon/N)T_i$:
    \begin{align*}
      K_X+A+S_1 &= K_X+A^\prime + S_1 +D =
      K_X+A^\prime+\bigl(1-\frac{\varepsilon}{N}\bigr)S_1 + B_1 \\
      K_X+A+S_2 & =
      K_X+A^\prime+\bigl(1-\frac{\varepsilon}{N}\bigr)S_2 + B_2
      +\frac{\varepsilon}{N} \divi_X (\varphi). 
    \end{align*}
    Next, define a new characteristic system $D^\prime\colon \NN^2 \to
    \Div_\QQ X$ by
    \begin{align*}
      D^\prime (\mathbf{e}_1)&=K_X+A^\prime+
      \bigl(1-\frac{\varepsilon}{N}\bigr)S_1 + B_1\\
      D^\prime (\mathbf{e}_2)&=K_X+A^\prime+
      \bigl(1-\frac{\varepsilon}{N}\bigr)S_2 + B_2
    \end{align*}
    By construction, $D^\prime$ is $\QQ$-linearly equivalent to $D$
    and the ring $R(X;D^\prime)$ is klt, which proves the first part
    of the statement in this case. 
\end{proof}

\subsection{Proper birational morphisms}
\label{sec:prop-birat-morph}

Adjoint rings behave well under proper birational morphisms;
when working with the restriction of strictly dlt rings, it is useful to
blow up $X$ to simplify singularities in order to satisfy the assumptions
of the lifting lemma. 

\begin{lem}
  \label{lem:5}
  Let $X$ be nonsingular projective and $R(X;D)$ a plt big
  adjoint ring on $X$:
\[
D(\lambda)=r(\lambda)\bigl(K_X+\Delta (\lambda) \bigr)
\quad
\text{where}
\quad
\Delta (\lambda)= A+S+B(\lambda) .
\]
 Let $B \subset X$ be a snc divisor such that all  
 $\Supp B(\lambda)\subset B$.
  There is a proper birational morphism $f\colon Y \to X$, and a plt
  big adjoint ring $R(Y;D^\prime)$:
 \[
D^\prime (\lambda)=r(\lambda)\bigl(K_Y+\Delta^\prime (\lambda) \bigr)
\quad
\text{where}
\quad
\Delta^\prime (\lambda)= A^\prime+T+B^\prime(\lambda)
\]
 with $T\subset Y$ the proper transform of $S$, with the following properties:
 \begin{itemize}
 \item The $f$-exceptional set $E$ is a divisor. Also, denoting by
   $B^\prime\subset Y$ the proper transform of $B\subset X$, $B^\prime
   \cup E$ is a snc divisor and all $\Supp B^\prime(\lambda) \subset
   B^\prime \cup E$;
 \item for all $\lambda \in \Lambda$, 
\[
K_Y+T+A^\prime+B^\prime(\lambda)=f^\star \Bigl(K_X+S+A+B(\lambda) \Bigr)
+ E(\lambda) ;
 \]
 where $E(\lambda)\geq 0$ is $f$-exceptional. This implies
 that $R(X;D)=R(Y;D^\prime)$, and;
\item for all $\lambda \in \Lambda$, the pair 
$\bigl(T,B^\prime (\lambda)_{|T}\bigr)$ is terminal. 
 \end{itemize}
\end{lem}

\begin{proof}[Proof in a special case]
  I prove the statement in a special case that contains all the ideas:
  assume that $\Lambda = \NN$, that is:
\[
R(X;D)=\bigoplus_{n\geq 0} H^0\bigl(X;n(K_X+S+A+B)\bigr)
\]
 where $(X,S+B)$ is a plt pair. 

 For $f\colon Y\to X$ a proper birational morphism with exceptional
 divisors $E_i$, I write
\[
K_Y +T + f^{-1}_\star B= f^\star(K_X+S+B)+\sum a_i E_i
\quad
\text{with all}
\quad
a_i>-1 .
\] 
Here $a_i=a(E_i; K_X+S+B)$ is the \emph{discrepancy} along $E_i$ of
the divisor $K_X+S+B$: it only depends on the \emph{geometric valuation} $\nu
= \nu (E_i)$ measuring order of vanishing along $E_i$. Next, setting
$B_Y=f^{-1}_\star B -\sum_{a_i<0}a_i E_i$, I get:
\[
K_Y+T+B_Y=f^\star(K_X+S+B)+E
\]
where $f_\star B_Y=B$ and $E=\sum_{a_\geq 0}a_iE_i \geq 0$ is
exceptional.

Pick a good resolution $f\colon Y \to X$ with the property that all
geometric valuations $\nu$ with $a(\nu ; K_X+S+B)<0$ are divisors on
$Y$ (the set of these valuations is finite hence such a resolution
exists); it is a simple matter to check that the pair $(T, B_{Y|T})$
is terminal.

Finally, choose an ample $\QQ$-divisor 
\[
A^\prime = f^\star A -\sum \varepsilon_i E_i
\]
on $Y$, where $0<\varepsilon_i \ll 1$. Setting $B^\prime=B_Y+\sum
\varepsilon_i E_i$, it is still true that $(T,B^\prime_{|T})$ is
terminal, and:
\[
K_Y +T+A^\prime +B^\prime =f^\star (K_X+S+A+B) + E .
\]
\end{proof}

\section{The main construction}
\label{sec:reduct-strictly-lt-1}

In this section I give a complete proof of theorem~A assuming
theorem~\ref{thm:12}.

\begin{lem}
  \label{lem:10}
  Let $X$ be a nonsingular projective variety, $\sum_{i=1}^r B_i$ a snc
  divisor on $X$, and $B=\sum_{i=1}^r b_i B_i$ a klt divisor
  (that is, $0\leq b_i<1$ for $i=1,\dots,r$).

  Let $A$ be an ample $\QQ$-divisor on $X$ and assume that for some
  integer $p>0$ $|p(K_X+A+B)|\neq 0$.

  Consider the parallelepiped:
\[
\Bc = \prod_{i=1}^r [b_i,1]B_i \subset \Div_\RR X,
\]
and the cone and monoid:
\[
\Cc = \RR_+ \bigl(K_X+A+\Bc \bigr) \subset \Div_\RR X;\quad
\Lambda =\Cc \cap \Div_\ZZ X.
\]
  
  Then, assuming theorem~\ref{thm:12}, the ring $R(X;\Lambda)$ is
  finitely generated.
\end{lem}

\begin{proof}
  In the course of the proof, I work in the vector subspace $V=\RR^r
  \subset \Div_\RR X$ spanned by the prime divisors $B_i$; I denote by
  $D\colon V \hookrightarrow \Div_\RR X$ the canonical inclusion.  By
  suitably enlarging (lemma~\ref{lem:2}) the set $\{B_i\}$, and an
  appropriate choice of the canonical divisor (lemma~\ref{lem:4}), I
  can assume that
\[
K_X+A+\sum_{i=1}^r b_iB_i = \sum_{i=1}^r p_i B_i
\]
where all $p_i\geq 0$. In addition, perhaps by blowing up $X$ and using
lemma~\ref{lem:5}, I can assume that, even after enlargement, the
divisor $\sum_{i=1}^r B_i$ is still snc.

The key is to ``chop up'' $R=R(X;\Lambda)$ into finitely many strictly
dlt subrings. Consider the $r$ `back faces' of the parallelepiped $\Bc$:
\[
\Bc_j=\{B=B_j+\sum_{i\not = j} c_i B_i\mid \text{all}\; b_i\leq c_i \leq
1\} 
\quad
\text{for}\;
j=1, \dots, r.
\]
It is clear that 
\[
\Cc=\bigcup_{j=1}^r\Cc_j, \quad
\text{where}
\quad
\Cc_j = \RR_+ \bigl(K_X+A+\Bc_j \bigr)
\]
(this uses in a crucial way that $\Cc \subset \sum_{i=1}^r \RR_+ B_i$)
hence, setting $\Lambda_j = \Cc_j \cap \ZZ^r$, $R =\sum_{j=1}^r
R_j$ where each $R_j=R(X;\Lambda_j)$ is a strictly dlt
adjoint ring. For $j=1,\dots, r$, denote by
\[
\rho_j\colon R_j\to R_{B_j}=R_{B_j}(X;\Lambda_j)
\] 
the surjective ring homomorphisms to the restricted rings. By
theorem~\ref{thm:12}, and by induction on dimension, the
$R_{B_j}$ are finitely generated. 

I show that $R$ is finitely generated. I can't prove this directly, so
let $\sigma_j\in H^0(X;B_j)$ be a section vanishing on $B_j$
($\sigma_j$ is determined up to multiplication by a nonzero constant);
I show instead that the ring
\[
R [\sigma_1, \dots, \sigma_r], 
\quad
\text{graded by}
\quad
\NN^r,
\]
is finitely generated. By lemma~\ref{lem:2} again, this implies that
$R$ itself is finitely generated. Fix the total degree function
\[
\tau \colon \NN^r \to \NN,
\quad
\tau (m_1, \dots, m_r)=\sum_{j=1}^r m_j .
\] 
Let $N\gg 0$ be large enough that the
following holds: 
\begin{multline}
  \label{eq:total_degree}
\text{If}\;
\mb=(m_1, \dots, m_r)\in \Cc_j\cap \ZZ^r
\;
\text{and}
\;
\tau (\mb) >N, \\
 \text{then}
\;
\mb -B_j = (m_1, \dots, m_j-1, \dots, m_r)\in \Cc \cap \ZZ^r.  
\end{multline}
(It is pretty obvious that you can find $N\gg 0$ with this property;
draw a picture!) Prepare now the following finite sets:
\begin{itemize}
\item A basis $G_0$ of $\oplus_{\tau(\lambda)\leq N}R_\lambda$;
\item For all $j=1, \dots, r$, a set $G_j\subset R_j\subset R$ lifting a
  set of generators of $R_{B_j}$.
\end{itemize}
I conclude the argument by showing that, using
equation~\ref{eq:total_degree}, the union $G=\cup_{j=0}^r G_j \cup
\{\sigma_0, \dots, \sigma_r\}$ generates the ring $R[\sigma_1,
\dots, \sigma_r]$. It is enough to show that $R\subset \CC[G]$.
Assume by induction that $M\geq N$ and: for $\tau (\lambda)\leq M$ all
$R_\lambda \subset \CC[G]$; let $\tau (\lambda)=M+1$, and $x
\in R_\lambda$.  Now, for some $j=1, \dots, r$, $\lambda \in \Cc_j$,
so consider the restriction homomorphism:
\[
\rho_j \colon R_j \to R_{B_j}  .
\]
It is clear that there is $x_j\in \CC[G_j]$ such that
$\rho_j(x-x_j)=0$. This means that 
\[
x-x_j = \sigma_j\, y
\]
where $y\in H^0(X; D(\lambda -B_j))$. By
property~\ref{eq:total_degree}, $\lambda-B_j\in \Cc$,
therefore 
\[
y\in R_{\lambda-B_j}
\quad
\text{has total degree}
\quad
\tau(y)=\tau (x)-1 ,
\]
hence, by induction, $y\in \CC[G]$, and hence also $x\in \CC[G]$. 
\end{proof}

\begin{proof}[Proof of theorem~A] I need to show that a big klt
  adjoint ring $R=R(X;D)$, where $D(\Cc)\subset \Div^+_\RR X$, is
  finitely generated. By a simple application of lemma~\ref{lem:2}, I
  may assume that $D\colon (\Cc \subset \RR^r) \to \Div_\RR^+ X $ is
  rational linear. By lemma~\ref{lem:9}, I may also assume that $D$ is
  injective. As before for $\mathbf{v}\in \Cc$ I write
  \[D(\mathbf{v})=r(\mathbf{v})\bigl(K_X+A+B(\mathbf{v})\bigr)=
  r(\mathbf{v})\Bigl(K_X+A+\sum_{i=1}^r b_i(\mathbf{v}) B_i\Bigr).\]

  Let $\eb_l$ for $l=1,\dots, m$ be generators of the cone 
  $\Cc$. I will shortly need the quantity
  \[
  \delta= \min_{i=1,\dots, r} \min_{\mathbf{v}\in \Cc}
  \{1-b_i(\mathbf{v})\}
   = \min_{i=1,\dots, r} \min_{l=1,\dots, m}
  \{1-b_i(\mathbf{e}_l)\} >0.
  \]
  
  For a rational vector $\mathbf{v}\in \Cc$ lying on the hyperplane
  $\Pi= \{\mathbf{v} \mid r(\mathbf{v})=1\}$, consider the
  parallelepiped $\Bc(\mathbf{v})=\prod [b_i(\mathbf{v}),1]\subset V$
  and the cone
  $\Cc(\mathbf{v})=\RR_+\bigl(K_X+A+\Bc(\mathbf{v})\bigr)$. Because
  all sides of all parallelepipeds $\Bc(\mathbf{v})$ have length $\geq
  \delta$, there are finitely many vectors $\mathbf{v}_1,
  \dots,\mathbf{v}_n \in \Cc \cap \Pi$ such that $\Cc \subset
  \cup_{k=1}^n \Cc (\mathbf{v}_k)$.  Write
  $\Lambda(\mathbf{v})=\Cc(\mathbf{v})\cap \ZZ^r$. Now $R\subset
  \sum_{k=1}^n R(X;\Lambda(\mathbf{v}_k))$ is finitely generated by
  lemma~\ref{lem:2} and lemma~\ref{lem:10}.
\end{proof}

\section{Lifting lemmas}
\label{sec:lifting-lemma}

A quick internet search will turn up several papers on lifting lemmas.
The prototype can be traced back to \cite{MR1660941}; the best place
to start learning the material is \cite[Theorem~11.5.1]{MR2095472};
the lifting theorem~\ref{thm:6}, theorem~\ref{thm:1} and
corollary~\ref{cor:2} are all due to Hacon and
M\textsuperscript{c}Kernan.

\subsection{General initial set-up}
\label{sec:general-set-up-1}

All variants and improvements of the lifting lemma have a common
initial set-up that I now summarise:

\begin{itemize}
\item $X$ is nonsingular projective; $S\subset X$ is a nonsingular 
  divisor;
\item $(X, \Delta =S+A+B)$ is a plt pair; here $A$ is an ample
  $\QQ$-divisor; I always assume that $A$ meets transversally
  everything in sight, and I sometimes assume that the coefficients of
  $A$ are sufficiently small;
\item Write $\Omega := (A+B)_{|S}$; I assume that the pair $(S,
  \Omega)$ is terminal.
\end{itemize}

The purpose of the lifting lemma is always this: Fix a strictly
positive integer $p$ such that $p\Delta \in \Div_\ZZ X$, then study
the \emph{restricted adjoint linear system}:
\begin{equation*}
  \label{eq:res_linear_sys}
|p(K_X+\Delta)|_{S}  .
\end{equation*}
Now, of course, $p(K_X+\Delta)_{|S}=p(K_S+\Omega)$, and, in general,
I don't expect the restricted linear system to be the complete linear
system $|p(K_S +\Omega)|$. Indeed, simple examples show that, for
basic reasons, the restricted linear system can have a fixed part.
The key point of the lifting lemma is to choose divisors $\Theta,
\Phi$ on $S$ with 
\[ 
0\leq \Theta \leq \Omega
\quad
\text{and}
\quad
\Theta +\Phi =\Omega
\]
and compare the restricted linear system $|p(K_X+\Delta)|_S$ with the
linear system with fixed part $|p(K_S+\Theta)|+\Phi$.

\begin{nota}
  Let $E_i$ be prime divisors on $X$. For divisors
\[
D_1=\sum d_1^i E_i, \quad D_2=\sum d_2^i E_i,
\]
  I write
\[
D_1 \wedge D_2 =\sum \min \{d_1^i,d_2^i\} E_i.
\]
\end{nota}

\subsection{Simple lifting}
\label{sec:easy-lifting}

This is the simplest statement that one can make:
\begin{thm}
  \label{thm:6} Fix an integer $p>0$ such that $p\Delta$ is integral.

  Assume that $S\not \subset \Bb(K_X+\Delta)$. 

  Write
\[
F_p=\frac1{p}\Fix |p(K_X+\Delta)|_{S}; 
\quad
\Phi_p = \Omega \wedge F_p; 
\quad
\Theta_p = \Omega -\Phi_p .
\]
Then
\[
|p(K_X+\Delta)|_{S}=|p(K_S+\Theta_p)|+p\Phi_p .
\]
\qed
\end{thm}

\subsection{Sharp lifting}
\label{sec:sharp-lifting}

Next I state a subtle but crucial improvement of the lifting lemma.

\begin{dfn}
  \label{dfn:8}
  Let $X$ be normal projective, $S\subset X$ a codimension $1$
  subvariety, and $D$ a $\QQ$-Cartier divisor on $X$. If $S \not \subset
  \Bb (D)$, then the \emph{restricted asymptotic fixed part} is
\[
\Fb_S (D) =\inf_{0<n\in \ZZ\,, nD \in \Div_\ZZ X}\; 
\frac1{n}\Fix\bigl( |nD|_{S}\bigr) \in \Div^+_\RR S.
\]
It is important to appreciate that, in general, even though $D$ is
a divisor with rational coefficients, $\Fb_S (D)$ can have nonrational
real coefficients.
\end{dfn}

\begin{rem}
  \label{rem:9}
  It is crucial to be aware that $\Fb_S(D)$, $\Fb (D)_{|S}$, and $\Fb
  (D_{|S})$ are three distinct divisors in general. 
\end{rem}
\begin{thm}  
  \label{thm:1} Fix an integer $p>0$ such that $p\Delta$ is integral.

  Assume that \[S\not \subset \Bb (K_X+\Delta+A/p)  .\]
(Note that this holds in particular if $S\not\subset \Bb
(K_X+\Delta+\varepsilon A)$ for some rational $0\leq \varepsilon \leq
1/p$.)
 
Write $\Fb_S = \Fb_S (K_X+\Delta+A/p)$. Consider a $\QQ$-divisor
$\Phi$ on $S$ such that
\[
p\Phi \quad \text{is integral and}
\quad
\Omega \wedge \Fb_S \leq \Phi \leq \Omega;
\quad
\text{write}
\quad \Theta = \Omega -\Phi .
\]
Then 
\[
|p(K_X+\Delta)|_{S}\supset |p(K_S + \Theta)|+p\Phi .
\]
\qed
\end{thm}

\begin{rem}
  \label{rem:1}
  Sharp lifting improves simple lifting in two ways: it relaxes the
  assumption and it strengthens the conclusion.
  \begin{description}
  \item[It relaxes the assumption] Here I just require that 
\[S \not \subset
  \Bb \bigl(K_X+\Delta + A/p\bigr) .\]
\item[It strengthens the conclusion] The conclusion now 
  holds for $\Omega
  \wedge \Fb_S \leq \Phi$ whereas earlier I required $\Omega \wedge
  F_p= \Phi_p$: note that 
  \begin{multline*}
\Fb_S=\Fb_S(K_X+\Delta +A/p)\leq \Fb_S(K_X+\Delta) \leq\\
\leq \frac1{p} \Fix \bigl( \left| p(K_X+\Delta)\right|_{S} \bigr)=F_p ,   
  \end{multline*}
  hence $\Phi$ is allowed potentially to be smaller than $F_p$. 
  \end{description}
 \end{rem}

\subsection{Tinkering lifting}
\label{sec:tinkering-lifting}

It is possible still to tinker with the statement of the lifting
lemma:

\begin{cor}
  \label{cor:2}
  Fix an integer $p>0$ such that $p\Delta$ is integral.

  Assume that $S\not \subset \Bb (K_X+\Delta)$.
Write $\Fb_S = \Fb_S (K_X+\Delta)$, and fix a rational $\varepsilon >0$
such that $\varepsilon (K_X+\Delta)+A$ is ample.

 Consider a $\QQ$-divisor $\Phi$ on $S$ such that
\[
p\Phi \quad \text{is integral and}
\quad
\Omega \wedge \Bigl(1-\frac{\varepsilon}{p}\Bigr)\Fb_S \leq \Phi \leq \Omega;
\quad
\text{write}
\quad \Theta = \Omega -\Phi .
\]
Then
\[
|p(K_X+\Delta)|_{S}\supset |p(K_S + \Theta)|+p\Phi .
\]
\end{cor}

\begin{proof}
  Corollary~\ref{cor:2} follows from theorem~\ref{thm:1}:
\[
K_X+\Delta+\frac1{p}A=\Bigl(1-\frac{\varepsilon}{p}\Bigr)(K_X+\Delta)+
\frac1{p}\Bigl( \varepsilon(K_X+\Delta)+A\Bigr),
\] 
 hence:
 \begin{multline*}
\Fb_S \Bigl(K_X+\Delta+\frac1{p}A \Bigr)
\leq \Bigl(1-\frac{\varepsilon}{p}\Bigr) \Fb_S(K_X+\Delta)
+\frac1{p}\Fb_S (\text{ample})=\\
=\Bigl(1-\frac{\varepsilon}{p}\Bigr) \Fb_S(K_X+\Delta) ;   
 \end{multline*}
thus, the assumptions (those pertaining to $\Phi$) of
theorem~\ref{thm:1} are satisfied, hence its conclusion holds.
\end{proof}

\section{Restriction of strictly dlt rings}
\label{sec:restr-strictly-lt-1}

In this final section, I sketch the proof of theorem~\ref{thm:12}. Let
me briefly recall the set-up. $X$ is nonsingular projective,
$\dim X = n$, and $R(X;D)$ is a strictly dlt big adjoint ring on $X$
with characteristic system:
\[
D(\lambda)=r(\lambda)\bigl(K_X + \Delta(\lambda) \bigr)
\quad 
\text{where}
\quad
\Delta(\lambda)=S+A+B(\lambda) . 
\]
The aim, remember, is to show that the restricted
ring
\begin{multline*}
R_S(X;D)=\bigoplus_{\lambda \in \Lambda}\Image (\rho_\lambda), \quad 
\text{where}\\
\rho_\lambda \colon
H^0\Bigl(X;r(\lambda)\bigl(K_X+\Delta(\lambda)\bigr)\Bigr) \to
H^0\Bigl(S;r(\lambda)\bigl(K_S+\Omega(\lambda)\bigr)\Bigr) 
\end{multline*}
is the restriction map, is a klt adjoint ring.

After some simple manipulations, I may in addition assume the following:

\begin{enumerate}
\item All $\bigl(S, B(\lambda)_{|S}\bigr)$ are terminal pairs. (This
  can be achieved by using lemmas~\ref{lem:4} and~\ref{lem:5} in
  tandem.)
\item $A$ has small coefficients and meets everything in sight as
  generically as possible; in particular, for instance, I assume that
  all pairs
\[
\Bigl(S,\Omega(\lambda)=\bigl(A+B(\lambda)\bigr)_{|S}\Bigr)
\]
  are terminal. (See remark~\ref{rem:8} for this.)
\item For $\lambda \in \Lambda$, $S\not \subset \Bb \bigl(D(\lambda)
  \bigr)$. This can be achieved as an application of theorem~$B_n$:
  $\Cc^\prime = \Cc \cap \RR_+ \bigl(\Bc_{V,A}^{S=1}\bigr)$ is a
  finite rational cone; now work with $\Lambda^\prime=\Lambda \cap
  \Cc^\prime$ and $R(X; \Lambda^\prime)$ in place of $\Lambda$ and
  $R(X,\Lambda)$: the point is that $R_S(X; \Lambda^\prime)=R_S(X;
  \Lambda)$.
\item Denote by $\eb_i\in \RR^r$ the standard basis vectors.  Then
  $\Cc=\sum_{i=1}^r \RR_+ \eb_i\subset \RR^r$ is a simplicial cone,
  $\Lambda = \NN^r\subset \RR^r$, and $D\colon \Lambda \to \Div_\QQ
  X$ is the restriction of a linear function that, abusing notation, I
  still denote by $D\colon \RR^r \to \Div_\RR X$. This can be
  achieved by finding a triangulation of $\Cc$ on which $D$ is linear.
 \end{enumerate}

\begin{nota}
   Below I denote by $\Pi \subset \RR^r$ the affine hyperplane spanned
   by the basis vectors $\eb_i$.
   
   By what I said, $\Delta, B \colon \Lambda \to \Div_\QQ^+ X$ are
   restrictions of functions that, abusing notation, I still denote by
   $\Delta, B \colon \RR^r \to \Div_\RR X$. These are degree $0$
   homogeneous; hence, they are determined by their restrictions to
   the affine hyperplane $\Pi\subset \RR^r$; note that these
   restrictions are affine.

   Similarly, $\Omega\colon \RR^r\to \Div_\RR S$ is degree $0$
   homogeneous, and $\Omega_{|\Pi}$ is affine. 
\end{nota}

\begin{lem}
  \label{lem:7}
  For $\lambda \in \Lambda$, write $\Fb_S(\lambda)=\Fb_S
 \bigl(D(\lambda)\bigr)$ (N.B. by construction if $\ZZ\ni n>0$, then
 $\Fb_S(n\lambda)=n\Fb_S(\lambda)$). 

 Then $\Fb_S (-)$ can be uniquely extended to a degree $1$ homogeneous
 convex function that, abusing notation, I still denote by
\[
\Fb_S \colon \Cc \to \Div_\RR^+ S .
\]
 $\Fb_S$ is continuous on the interior $\Interior \Cc$ but not
 necessarily on $\Cc$.
 \end{lem}

 \begin{proof}
   By homogeneity I extend to $\Fb_S \colon \Cc \cap \QQ^r \to
   \Div_\RR^+ S$; this function is homogeneous convex hence locally
   Lipschitz hence locally uniformly continuous hence it can uniquely
   be extended to a function on $\Cc$ continuous on $\Interior \Cc$.
 \end{proof}

After some further blowing up, I may in addition assume:

\begin{enumerate}
\item[5.] there is a fixed snc divisor $F$ on $S$ such that, for all
  $\lambda \in \Lambda$, $\Supp \Fb_S(\lambda) \subset F$.
\end{enumerate}

Now for $\wb \in \Cc$ write: 
\[
\Omega(\wb)=\bigl(A+B(\wb)\bigr)_{|S};
\quad
\Phi(\wb)=\Omega(\wb)\wedge\Fb_S(\wb);
\quad
\Theta(\wb)=\Omega(\wb)-\Phi(\wb) .
\]

By construction, for all $\wb \in \Cc$, $0\leq
\Theta(\wb)\leq \Omega(\wb)$ and
$\Theta(\wb)+\Phi(\wb)=\Omega(\wb)$. 

\begin{lem}
  \label{lem:6}
 For all $\lambda \in \Lambda$, there is $\ZZ\ni n=
   n(\lambda)>0$ such that
\[
\Phi (\lambda) = \Omega(\lambda)\wedge
\frac1{n} \Fix \bigl(|D(n\lambda)|_{S} \bigr) .
\]
In particular, for all $\lambda \in \Lambda$, $\Theta (\lambda)\in
\Div_\QQ^+ S$ is a rational divisor (and so is $\Phi (\lambda)$).
\end{lem}

\begin{proof}
  The proof is explained very well in \cite[Theorem~7.1]{Hc}; it is an
  application of tinkering lifting; it is simpler than and based on
  the same idea of the proof of lemma~\ref{lem:3} below.
\end{proof}

The proof of theorem~\ref{thm:12} follows easily if I show that
\[
\Theta_{|\Pi\cap \Cc}\colon \Pi\cap \Cc \to \Div_\RR^+ S
\] 
is \textbf{piecewise affine}. (Indeed, write
$D_S(\lambda)=r(\lambda)\bigl(K_S+\Theta (\lambda) \bigr)$. By
lemma~\ref{lem:6} and sharp lifting, the restricted ring $R_S(X;D)$
and the adjoint ring $R(S,D_S)$ have a common Veronese subring.)  I don't
prove the statement completely.  Instead, in the remaining part of
this section, I prove:

\begin{lem}
  \label{lem:3}
  Let $\xb\in \Pi(\RR)$ and assume that the smallest rationally defined
  affine subspace $U\subset \RR^r$ containing $\xb$ is $\Pi$.
  Then, $\Theta_{|\Pi\cap \Cc}$ is affine in a neighbourhood of $\xb$.
\end{lem}

This is compelling, but note that it stops short of proving
theorem~\ref{thm:12}: the statement implies that there is a
decomposition of $\Cc$ in rational subcones such that $\Theta$ is
affine on each subcone, but there is no guarantee that the
decomposition is locally finite, nor indeed that the subcones
themselves are finite. The proof of the lemma contains
all the ideas of Lazi\'c's proof of theorem~\ref{thm:12}. 

\begin{lem}[Diophantine approximation]
  \label{lem:1}
  Let $\xb\in \RR^n$; denote by $U$ the smallest rationally defined
  affine subspace containing $\xb$, and let $\dim U=m-1$. 
  
  Fix $\varepsilon>0$ and an integer $M>0$. There exist vectors
  $\wb_1, \dots, \wb_m \in U (\QQ)$ with the following properties:
  \begin{enumerate}
  \item For $i=1, \dots, m$, there are real numbers
\[
0<r_i<1\quad \text{with}\quad
\sum_{i=1}^m r_i=1, \quad 
\xb = \sum_{i=1}^m r_i\wb_i ;
\]
  \item there is a $m$-tuple $(p_1,\dots, p_m)$ of
  strictly positive integers, all $p_i$ divisible by $M$, such that
  all $p_i\wb_i\in \ZZ^n$ are integral, and
\[
|\!| \xb-\wb_i |\!|<\varepsilon/p_i  .
\] 
  \end{enumerate} \qed
\end{lem}

\begin{proof}[Proof of lemma~\ref{lem:3}]
There is a nagging difficulty with the proof:

\subsubsection*{A nagging difficulty and an additional assumption}
\label{sec:small-difficulty-an}
The point is this: by definition,
\[
\Phi (-)=\Omega (-)\wedge \Fb_S (-)\quad
\text{and}
\quad
\Theta (-)=\Omega (-)-\Phi (-);
\]
hence, although $\Omega\colon \Cc \to \Div_\RR^+ S$ is linear, and
$\Fb_S \colon \Cc \to \Div_\RR^+ S$ is convex, $\Phi (-)$ is not
necessarily convex, and $\Theta (-)$ is not necessarily concave. To be
more specific, if for some prime divisor $P\subset S$, $\mult_P \Omega (\xb)
= \mult_P \Fb_S (\xb)$, then $\Theta (-)$ may fail to be concave in a
neighbourhood of $\xb$. I first run the proof under the following
\emph{additional assumption}:
\[
\text{For all prime divisors}\;
P\subset S, 
\quad 
\mult_P \Omega (\xb) \not = \mult_P \Fb_S (\xb).
\]
In the proof below, I only use the additional assumption to ensure
that $\Theta (-)$ is concave in a neighbourhood of $\xb$.  
At the end, I briefly explain how to get around this difficulty.

\subsubsection*{The strategy of the proof}
\label{sec:strategy-proof}

The idea of the proof is to choose a real $\varepsilon >0$, an integer
$M>0$ and a rational Diophantine approximation
\[
(\wb_i,\Theta_i)\in
\Pi(\QQ)\times \Div_\QQ (S)\quad
\text{of the vector}
\quad
\bigl(\xb,\Theta(\xb)\bigr)\in \Pi(\RR)\times \Div_\RR(S) ,
\]
such that
\begin{enumerate}
\item[i.] For $i=1,\dots,r$, there are real numbers 
\[0<\mu_i<1
\quad 
\text{with}
\quad \sum_{i=1}^r\mu_i=1, \;\text{and}\;
\begin{cases}
\xb & = \sum_{i=1}^r \mu_i \wb_i ,\\
\Theta (\xb) & =\sum_{i=1}^r \mu_i \Theta_i . 
\end{cases}
\]
In particular this implies that $\Delta(\xb)=\sum \mu_i
\Delta(\wb_i)$. (All this is guaranteed by lemma~\ref{lem:1}.)  In
addition I require that:
\item[ii.] $\Theta_i \leq \Theta (\wb_i)$.
\end{enumerate}
Indeed, once I know this, then, by concavity of $\Theta$:
\[
\sum \mu_i \Theta (\wb_i) \leq \Theta (\xb) = \sum \mu_i \Theta_i  .
\]
I deduce $\Theta (\xb)=\sum \mu_i \Theta (\wb_i)$ and, by concavity
again, this implies that $\Theta$ is affine on the convex span
of the $\wb_i$.

\subsubsection*{Choice of $\varepsilon$}
\label{sec:choice-varepsilon}

I choose $\varepsilon>0$ small enough that it has the following 
\textbf{features}: 
\begin{enumerate}
\item[(a)] \label{eps:a} $\Theta$ is concave in $\{\wb \mid
  |\!|\wb-\xb|\!|<\varepsilon \}$. (This is OK by the additional assumption.)
\item[(b)] \label{eps:b} There is a local Lipschitz constant $C=C_\xb$
  such that:
  \begin{equation*}
    \label{eq:lipschitz}
\text{If}\; |\!|\wb-\xb|\!|<\varepsilon, \quad
\text{then}\; |\!|\Theta (\wb)-\Theta (\xb)|\!|<C|\!|\wb-\xb|\!|  .
  \end{equation*}
    (All concave functions are locally Lipschitz.)
\item[(c)] \label{eps:c} There is a constant $0<\delta<1$ with the
  following property: For all prime divisors $P\subset S$:
  \begin{multline*}
    \label{eq:delta}
\text{If}\; \mult_P \bigl( \Omega(\xb) - \Theta(\xb)\bigr)>0,\; 
\text{and}\; |\!|\wb-\xb|\!|<\varepsilon, \\
\text{then}\; \mult_P \bigl( \Omega(\wb) - \Theta(\wb)\bigr)>\delta,
  \end{multline*}
  and I also assume that $\varepsilon <\delta$.
\item[(d)] \label{eps:d} If $p\geq 1$ and
  $|\!|\wb-\xb|\!|<\frac{\varepsilon}{p}$, then the $\QQ$-divisor
\[
\Delta (\wb)-\Delta (\xb) + \frac{A}{p} 
\] 
is ample. 
\item[(e)] \label{eps:e} If $|\!|\wb-\xb|\!|<\varepsilon$, then the
  $\QQ$-divisor
\[
(C+1)\frac{\varepsilon}{\delta}\Bigl(K_X+\Delta (\wb) \Bigr) + A
\]
is ample.
\item[(f)] \label{eps:f} If $\Theta \in \Div_\QQ S$ and $|\!|\Theta
  -\Theta(\xb)|\!|<\varepsilon$, then no component of $\Theta$ is in
  the asymptotic fixed part
\[
\Fb (K_S+\Theta) .
\]
It is a simple consequence of Theorem~$B_{n-1}^+$ that I can arrange
for this to hold. 
\end{enumerate}

\subsubsection*{Choice of $M>0$}
\label{sec:choice-dioph-appr}

Next I choose $M>0$ such that the rational Diophantine approximation
given by lemma~\ref{lem:1}:
\[
(\wb_i,\Theta_i)\in
\Pi(\QQ)\times \Div_\QQ (S)\quad
\text{of the vector}
\quad
\bigl(\xb,\Theta(\xb)\bigr)\in \Pi(\RR)\times \Div_\RR(S) ,
\]
satisfies the following \textbf{conditions}:
\begin{enumerate}
\item \label{con:1} For $i=1,\dots,r$, there are real numbers 
\[0<\mu_i<1
\quad 
\text{with}
\quad \sum_{i=1}^r\mu_i=1, \;\text{and}\;
\begin{cases}
\xb & = \sum_{i=1}^r \mu_i \wb_i ,\\
\Theta (\xb) & =\sum_{i=1}^r \mu_i \Theta_i . 
\end{cases}
\]
 In particular this implies that $\Delta(\xb)=\sum \mu_i \Delta(\wb_i)$.
\item \label{con:2} There is an $r$-tuple $(p_1,\dots,p_r)\in \NN^r$
  of positive integers such that:
  \begin{itemize}
  \item $(p_i\wb_i;p_i\Theta_i)\in \NN^r \times \Div_\ZZ (S)$ is integral;
  \item 
\[\text{For all i,}\quad
|\!|\xb -\wb_i|\!|<\frac{\varepsilon}{p_i}
\quad
\text{and}
\quad
|\!|\Theta(\xb)-\Theta_i |\!| <
  \frac{\varepsilon}{p_i}  .
\] 
  \end{itemize}
\item \label{con:3} For all prime divisors $P\subset S$:
  \begin{itemize}
  \item If $\mult_P \Theta(\xb)< \mult_P \Omega(\xb)$, then also
    $\mult_P \Theta_i < \mult_P \Omega (\wb_i)$ (this is automatic from
    feature~(c));
  \item If $\mult_P \Theta(\xb)=\mult_P \Omega(\xb)$, then also
  $\mult_P \Theta_i = \mult_P \Omega (\wb_i)$.
  \item If $\mult_P \Theta(\xb)=0$, then also $\mult_P \Theta_i=0$.
  \end{itemize}
  (Although the second bullet point doesn't strictly speaking follow
  from a blind usage of lemma~\ref{lem:1}, it is easy to arrange for
  it to hold.  Indeed in this case $\mult_P \Theta (\xb)=\sum \mu_i
  \mult_P \Omega(\wb_i)$ and it pays to declare from the start that
\[
 \mult_P \Theta_i = \mult_P \Omega (\wb_i) .
\]
  The third bullet point is similar and easier.) 
\item \label{con:4} $M$ is large enough that the $p_i$ are large
  enough and divisible enough that:  
\[
\Fb (K_S+\Theta_i) = \Fix \left|p_i(K_S+\Theta_i)\right| .
\] 
 (This can easily be arranged using theorem~$B_{n-1}^+$.)
 \end{enumerate}

\subsubsection*{The key inclusion}
\label{sec:key-inclusion}

For all $i=1,\dots,r$ I show the \emph{key inclusion}: 
\begin{equation}
  \label{eq:key_inclusion}
  \left|p_i\Bigl(K_X+\Delta(\wb_i)\Bigr)  \right|_{S} 
  \supset\left| p_i\Bigl(K_S+\Theta_i \Bigr)\right| +
  p_i\Bigl(\Omega(\wb_i)-\Theta_i \Bigr)  .
\end{equation}
The key inclusion allows me to control the restricted algebra in a
neighbourhood of $\xb$: as I show below, it readily implies that
$\Theta_i \leq \Theta (\wb_i)$. If you get bored with the details of
the proof, you may want to press forward to the conclusion.

I plan to prove this using sharp lifting. To begin with, I remark that
I am in the \emph{general initial set-up} of
section~\ref{sec:general-set-up-1}.

For all $i$, I now check that the specific assumptions of sharp lifting are
satisfied; that is:
\begin{equation*}
  \label{eq:need_to_check}
  \Omega(\wb_i) \wedge \Fb_i \leq \Omega(\wb_i)-\Theta_i 
\end{equation*}
where $\Fb_i= \Fb_S \bigl(K_X+\Delta(\wb_i)+A/p_i\bigr)$.
For all prime divisors $P\subset S$ I check that
\begin{equation*}
  \label{eq:need_to_check_P}
\mult_P \bigl(\Omega(\wb_i) \wedge \Fb_i\bigr)
\leq \mult_P \bigl(\Omega(\wb_i)-\Theta_i \bigr)  .
\end{equation*}
The discussion breaks down in two cases:

\paragraph{Case 1: $\mult_P \Theta (\xb) =\mult_P \Omega (\xb)$ .}
\label{sec:theta-xb=omega-xb}

By condition~\ref{con:3}, $\mult_P \Theta_i= \mult_P
\Omega(\wb_i)$. By feature~(d), $\Delta (\wb_i)-\Delta (\xb) + A/p_i$
is ample, therefore:
\begin{multline*}
\mult_P \Fb_S \Bigl(K_X+\Delta (\wb_i) +\frac{A}{p_i}\Bigr) =\\ 
=\mult_P\Fb_S \Bigl(K_X + \Delta (\xb)+ \Bigl(\Delta (\wb_i)-\Delta (\xb) +
\frac{A}{p_i}\Bigr)\Bigr)\leq \\   
\leq \mult_P\Fb_S \bigl(K_X + \Delta (\xb)\bigr)=0 .
\end{multline*}

\paragraph{Case 2: $\mult_P \Theta (\xb) <\mult_P \Omega (\xb)$ .}
\label{sec:theta-xb-}
Using feature~(e):
\begin{multline*}
\mult_P \Fb_i = \mult_P \Fb_S\Bigl(K_X+\Delta(\wb_i)+\frac{A}{p_i}\Bigr)\leq\\
\leq \Bigl(1-\frac{(C+1)}{p_i}\frac{\varepsilon}{\delta}\Bigr)\mult_P
\Fb_S \bigl(K_X + \Delta(\wb_i)\bigr)  =\\=
\Bigl(1-\frac{(C+1)}{p_i}\frac{\varepsilon}{\delta}\Bigr)
\mult_P \Fb_S (\wb_i) .
\end{multline*}
This implies that
\begin{multline*}
  \label{eq:stuff}
  \mult_P \Omega(\wb_i)\wedge \Fb_i \leq
  \Bigl(1-\frac{(C+1)}{p_i}\frac{\varepsilon}{\delta}\Bigr)
\mult_P \bigl(\Omega(\wb_i)-\Theta(\wb_i)\bigr)\\
(\text{indeed, by definition}\;
\Omega (\wb_i)-\Theta(\wb_i)=\Omega(\wb_i)\wedge \Fb_S (\wb_i)) ; 
\end{multline*}
and, finally, using feature~(c) and:
\begin{equation*}
|\!|\Theta(\wb_i)-\Theta_i |\!|
\leq |\!|\Theta (\wb_i)-\Theta (\xb)|\!|+|\!|\Theta (\xb)-\Theta_i|\!|
\leq
C\frac{\varepsilon}{p_i}+\frac{\varepsilon}{p_i}=(C+1)\frac{\varepsilon}{p_i}, 
\end{equation*}
in tandem, I get:
\begin{multline*}
  \label{eq:finally}
  \Bigl(1-\frac{(C+1)}{p_i}\frac{\varepsilon}{\delta}\Bigr)
\mult_P \bigl(\Omega(\wb_i)-\Theta(\wb_i)\bigr) \leq\\
\leq
\mult_P \bigl(\Omega(\wb_i)-\Theta(\wb_i)\bigr)-\frac{C+1}{p_i}
\varepsilon
\leq\\
\leq
\mult_P \bigl(\Omega(\wb_i)-\Theta(\wb_i)\bigr)
+\mult_P \bigl(\Theta(\wb_i) -\Theta_i\bigr)=\\
= \mult_P \bigl(\Omega(\wb_i) -\Theta_i\bigr)  .
\end{multline*}

\subsubsection*{Conclusion}
\label{sec:conclusion}

I show that the key inclusion of equation~\ref{eq:key_inclusion}
implies the statement. By construction of the function $\Theta \colon
\Cc \to \Div_\RR^+ S$ and the lifting lemma, I know that:
\[
\left|p_i\bigl(K_X+\Delta(\wb_i)\bigr)\right|_{S} = \left|
p_i\bigl(K_S + \Theta(\wb_i)\bigr)\right| + \Phi (\wb_i).
\]
Thus, the key inclusion readily implies that:
\begin{equation}
  \label{eq:getting_there}
  \Mob \left|p_i\bigl(K_S + \Theta_i \bigr)\right| \leq \Mob  \left|
    p_i\bigl(K_S + \Theta(\wb_i)\bigr)\right| .
\end{equation}
Now, by feature~(f), no component of $\Theta_i$ is in the
fixed part of $\left|p_i(K_X+\Theta_i) \right|$; thus, the last
equation implies
\[
\Theta_i\leq \Theta (\wb_i) .
\]
\end{proof}

\subsubsection*{How to remove the additional assumption}
\label{sec:how-remove-addit}

Assume that for some prime $P\subset S$ $\mult_P \Omega (\xb)=\mult_P
\Fb_S (\xb)$. Consider an effective divisor 
\[
G= \sum_{j=1}^r \varepsilon_j B_j .
\]
If the coefficients $0\leq \varepsilon_j$ are small enough, then:
\begin{itemize}
\item Writing $B^\prime (\lambda) = B(\lambda)+G>B(\lambda)$, all 
  the $\bigl(S, B^\prime(\lambda)_{|S}\bigr)$ are terminal, and:
\item $A -G$ is still ample, so I can choose $A^\prime \sim_\QQ A - G$
  meeting everything in sight transversally, and such that, upon setting 
  \[
  \Delta^\prime (\lambda)=S+A^\prime +B^\prime (\lambda);\quad
  \Omega^\prime (\lambda)=\bigl(A^\prime +B^\prime
  (\lambda)\bigr)_{|S} , 
  \] 
  then all the pairs $\bigl(S, \Omega^\prime (\lambda) \bigr)$
  are terminal.
\end{itemize}
 
Note that, from the definition, for all $\lambda \in \Lambda$:
\[
\Fb_S(\lambda)=\Fb_S \bigl(D(\lambda)\bigr)=\Fb_S \bigl(D^\prime
(\lambda) \bigr) .
\]
By choosing $A^\prime$ generically, I can arrange that the additional
assumption for $D^\prime$ is satisfied, and conclude as above that
$\Theta^\prime (-)$ is rational affine in a neighbourhood of $\xb$.
By construction:
\[
\text{In a neighbourhood of $\xb$},\quad
\mult_P
\bigl(\Omega^\prime (-)-\Theta^\prime (-)\bigr)=\mult_P \Fb_S(-) ,
\]
that is, in a neighbourhood of $\xb$, $\mult_P\Fb_S(-)$ also is
rational affine. But then
\[
\xb \in U=\{\wb \mid \mult_P\Omega(\wb)=\mult_P\Fb_S(\wb)\}
\]
implies that, in a neighbourhood of $\xb$,
$\mult_P\Omega(\wb)=\mult_P\Fb_S (\wb)$ ($U$ is affine and defined
over $\QQ$, hence $\Pi\subset U$ by minimality of $\Pi$), that is,
$\Theta (-)$ is concave in a neighbourhood of $\xb$ after all.  \qed

\bibliography{mmp.bib}

\newcommand{\etalchar}[1]{$^{#1}$}
\begin{thebibliography}{BCHM09}

\bibitem[ADHL10]{ivan0}
Ivan Arzhantsev, Ulrich Derenthal, Juergen Hausen, and Antonio Laface.
\newblock Cox rings, arXiv:1003.4229\setbox0=\hbox{2010}.

\bibitem[AH06]{MR2207875}
Klaus Altmann and J{\"u}rgen Hausen.
\newblock Polyhedral divisors and algebraic torus actions.
\newblock {\em Math. Ann.}, 334(3):557--607, 2006.

\bibitem[BCHM09]{Hb}
Caucher Birkar, Paolo Cascini, Christopher~D. Hacon, and James McKernan.
\newblock Existence of minimal models for varieties of log general type.
\newblock {\em J. Amer. Math. Soc.}, posted on November 13, 2009.
\newblock PII: S 0894-0347(09)00649-3 (to appear in print).

\bibitem[CL10]{CoLa}
Alessio Corti and Vladimir Lazi\'c.
\newblock Finite generation implies the {M}inimal {M}odel {P}rogram,
  arXiv:1005.0614\setbox0=\hbox{2010}.

\bibitem[Cor07]{MR2352762}
Alessio Corti, editor.
\newblock {\em Flips for 3-folds and 4-folds}, volume~35 of {\em Oxford Lecture
  Series in Mathematics and its Applications}.
\newblock Oxford University Press, Oxford, 2007.

\bibitem[ELM{\etalchar{+}}06]{MR2282673}
Lawrence Ein, Robert Lazarsfeld, Mircea Musta{\c{t}}{\u{a}}, Michael Nakamaye,
  and Mihnea Popa.
\newblock Asymptotic invariants of base loci.
\newblock {\em Ann. Inst. Fourier (Grenoble)}, 56(6):1701--1734, 2006.

\bibitem[FM00]{MR1863025}
Osamu Fujino and Shigefumi Mori.
\newblock A canonical bundle formula.
\newblock {\em J. Differential Geom.}, 56(1):167--188, 2000.

\bibitem[HM06]{MR2242631}
Christopher~D. Hacon and James McKernan.
\newblock Boundedness of pluricanonical maps of varieties of general type.
\newblock {\em Invent. Math.}, 166(1):1--25, 2006.

\bibitem[HM09]{Hc}
Christopher~D. Hacon and James McKernan.
\newblock Existence of minimal models for varieties of log general type {II}.
\newblock {\em J. Amer. Math. Soc.}, posted on November 13, 2009.
\newblock PII: S 0894-0347(09)00651-1 (to appear in print).

\bibitem[Kol92]{MR1225842}
J\'anos Koll\'ar, editor.
\newblock {\em Flips and abundance for algebraic threefolds}.
\newblock Soci\'et\'e Math\'ematique de France, Paris, 1992.
\newblock Papers from the Second Summer Seminar on Algebraic Geometry held at
  the University of Utah, Salt Lake City, Utah, August 1991, Ast{\'e}risque No.
  211 (1992).

\bibitem[Laz04]{MR2095472}
Robert Lazarsfeld.
\newblock {\em Positivity in algebraic geometry. {II}}, volume~49 of {\em
  Ergebnisse der Mathematik und ihrer Grenzgebiete. 3. Folge. A Series of
  Modern Surveys in Mathematics [Results in Mathematics and Related Areas. 3rd
  Series. A Series of Modern Surveys in Mathematics]}.
\newblock Springer-Verlag, Berlin, 2004.
\newblock Positivity for vector bundles, and multiplier ideals.

\bibitem[Laz09]{lazic09:_towar_mmp}
Vladimir Lazi\'c.
\newblock Adjoint rings are finitely generated,
  arXiv:0905.2707\setbox0=\hbox{2009}.

\bibitem[Nak04]{MR2104208}
Noboru Nakayama.
\newblock {\em Zariski-decomposition and abundance}, volume~14 of {\em MSJ
  Memoirs}.
\newblock Mathematical Society of Japan, Tokyo, 2004.

\bibitem[P{\u{a}}u08]{Paun}
Mihai P{\u{a}}un.
\newblock {R}elative critical exponents, non-vanishing and metrics with minimal
  singularities, arXiv:0807.3109\setbox0=\hbox{2008}.

\bibitem[Sho03]{MR1993750}
Vyacheslav~V. Shokurov.
\newblock Prelimiting flips.
\newblock {\em Tr. Mat. Inst. Steklova}, 240(Biratsion. Geom. Linein. Sist.
  Konechno Porozhdennye Algebry):82--219, 2003.

\bibitem[Siu98]{MR1660941}
Yum-Tong Siu.
\newblock Invariance of plurigenera.
\newblock {\em Invent. Math.}, 134(3):661--673, 1998.

\end{thebibliography}

\end{document}